 \documentclass[11pt,fleqn]{amsart}

\usepackage{amsmath,amssymb,dsfont}
\usepackage{enumerate}
\usepackage{dsfont}
\usepackage{mathrsfs} 
\usepackage{txfonts} 
\usepackage{amsbsy}  
\usepackage[round]{natbib}  
\usepackage{graphicx}
\usepackage[american,british]{babel} 
\usepackage[ansinew]{inputenc}
\usepackage{diagbox,multirow}

\newtheorem{lemma}{Lemma}

\newtheorem{theorem}[lemma]{Theorem}

\newcommand{\cid}{\stackrel{\mbox{\tiny $d$}}{\longrightarrow}} 
\newcommand{\cip}{\stackrel{\mbox{\tiny $p$}}{\longrightarrow}} 
\newcommand{\asc}{\stackrel{\mbox{\tiny $a.s.$}}{\longrightarrow}} 

\newcommand{\varind}[1]{\mathds{1}_{    #1    }  }	
\newcommand{\Ind}[1]{\mathds{1}\mbox{\Large$_{ \{    #1  \}  }$}}

\newcommand{\be}{\begin{equation}}
\newcommand{\ee}{\end{equation}}
\newcommand{\bea}{\begin{eqnarray*}}
\newcommand{\eea}{\end{eqnarray*}}
\newcommand{\ben}{\begin{enumerate}[a)]} 
\newcommand{\bEN}{\begin{enumerate}[1.]}
\newcommand{\bECR}{\begin{enumerate}[(I)]} 
\newcommand{\bESR}{\begin{enumerate}[(i)]} 
\newcommand{\een}{\end{enumerate}}
\newcommand{\eEN}{\end{enumerate}}
\newcommand{\bit}{\begin{itemize}}
\newcommand{\eit}{\end{itemize}}

\newcommand{\Var}{\mathrm{Var}}

\DeclareMathOperator*{\vectorize}{vec}

\DeclareMathOperator*{\argmin}{arg\,min}

\DeclareMathOperator*{\trace}{tr}
\renewcommand*{\vec}{\vectorize\!}

\def\Sn{S\!_n}
\def\Snstar{S\!_n^{\,*}}
\def\SSCM{SSCM} 

\def\hV{\hat{V}}

\def\tZ{\tilde{Z}}
\newcommand{\X}{\mathbb{X}}
\newcommand{\Y}{\mathbb{Y}}

\newcommand{\calA}{\mathcal{A}}

\newcommand{\calI}{\mathcal{I}}
\newcommand{\N}{\mathds{N}}  
\newcommand{\R}{\mathds{R}}

\begin{document}

\title{The spatial sign covariance matrix with unknown location}

\author[A. D\"urre]{Alexander D\"urre}
\author[D. Vogel]{Daniel Vogel}
\author[D. E. Tyler]{David E.~Tyler}

\address{Fakult\"at Statistik, Technische Universit\"at Dortmund, 44221 Dortmund, Germany.}
\email{alexander.duerre@tu-dortmund.de}
\address{Fakult\"at f\"ur Mathematik, Ruhr-Universit\"at Bochum, 44780 Bochum, Germany.}
\email{vogeldts@rub.de}
\address{Department of Statistics, Rutgers University, Piscataway, NJ 08854, USA.}
\email{dtyler@rci.rutgers.edu}


\begin{abstract}
The consistency and asymptotic normality of the spatial sign covariance matrix with unknown location are shown.
Simulations illustrate the different asymptotic behavior when using the mean and the spatial median as location estimator.   

\end{abstract}

\maketitle

\section{Introduction}
\label{sec:intro}

We define the \emph{spatial sign of $x \in \R^p$} as $s(x) = x/|x|$ for $x \neq 0$ and $s(0) = 0$, where $|\cdot|$ denotes the Euclidean norm in $\R^p$. Let $X$ be a $p$-dimensional random vector, $p \ge 2$, having distribution $F$. For $t \in \R^p$, we call 
\[
	S(F,t) = E \left\{ s(X-t) s(X-t)^T \right\}
\]
the \emph{spatial sign covariance matrix (SSCM) of the distribution $F$ (or random variable $X$) with location $t$}.  Letting further $\X_n = (X_1,\ldots,X_n)^T$, where $X_1,\ldots,X_n$ represents a random sample from the distribution $F$, we call
\[
		\Sn(\X_n,t)=\frac{1}{n}\sum_{i=1}^n s(X_i-t) s(X_i-t)^T
\] 
the \emph{spatial sign covariance matrix of the sample $\X_n$ with location $t$}. The term spatial sign covariance matrix was coined by \citet*{Visuri2000}, but the estimator has a longer history in the statistics literature. It has excellent robustness properties: its influence function is bounded, and the asymptotic breakdown point is $1/2$ \citep*{Croux2010}. Together with its simplicity, this makes the \SSCM\ a popular scatter estimator in multivariate data analysis. 
%
%

Within the theory of multivariate scatter estimation, affine equivariance plays an important role. Let $\Y_n = \X_n A^T + 1_n b^T$ denote the $n \times p$ data matrix obtained from $\X_n$ by applying the affine linear transformation $x \mapsto A x + b$ to each data point (where $1_n$ denotes the $n$-vector consisting of ones). An \emph{affine equivariant scatter estimator}, $V_n$ say, satisfies 
\be \label{eq:are}
	\hV_n(\Y_n) =  A \hV_n(\X_n) A^T
\ee
for any $b \in \R^p$ and any full rank square matrix $A$, i.e.\ it behaves like the covariance matrix under linear transformations of the data. The problem of robust, affine equivariant scatter estimation has received much attention in the last decades, see e.g. \citet*{Maronna2006} or \citet{Zuo2006} for an overview. However, the \SSCM\ lacks this property. It fulfills the weaker condition of \emph{orthogonal invariance}, that is, it satisfies (\ref{eq:are}) for all orthogonal matrices $A$. 
This is closely related to the fact that, at elliptical distributions, the \SSCM\ shares the eigenvectors (and the ranking of the eigenvalues) with the covariance matrix, but the exact connection between the eigenvalues of the \SSCM\ and covariance matrix is, even under ellipticity, rather tricky. An explicit expression is known only for $p = 2$ \citep[e.g.][]{Duerre2014}. Thus, the \SSCM\ gives information about the \emph{orientation} of the data \citep[cf.][]{Bensmail1996}, and its use has primarily been proposed for analyses that are based on this information only, most notably principal component analysis \citep{Marden1999, Locantore1999, Croux2002, Gervini2008}. Other such applications are direction-of-arrival estimation \citep{Visuri2001a} or testing for sphericity in the elliptical model \citep{Sirkia2009}. The latter makes use of the fact that under the null hypothesis that $X$ is spherical, $s(X)$ is uniformly distributed on the $p$-dimensional unit sphere.

The scatter estimator proposed by \citet{Tyler1987} can be regarded as an affine equivariant version of the \SSCM. It lacks, however, the high breakdown point and the simplicity of the \SSCM, its computation requiring an iterative algorithm. The estimate obtained when stopping the algorithm after a finite number of steps has been considered and called the $k$-step SSCM by \citet{Croux2010}. It keeps the high breakdown point of the SSCM, but has an asymptotic variance close to that of Tyler's estimator.

In this paper, we are concerned with the asymptotic properties of the \SSCM. 
The strong law of large numbers and the central limit theorem immediately yield
\begin{enumerate}[(I)]
\item \label{num:I}
$\Sn(\X_n,t) \asc S(F,t)$ \ and
\item  \label{num:II}
$\sqrt{n} \vec\, \{ \Sn(\X_n,t) - S(F,t)\} \cid N\!_{p^2}( 0,\, W)$, 
\end{enumerate}
for every distribution $F$ and $t \in \R^p$, where $W = \Var(Z)$ with $Z = Z(X,t) = \vec\, \{ s(X-t)s(X-t)^T \}$ and $X\sim F$.

However, the estimator $\Sn(\X_n,t)$ itself is rarely applicable, since the central location $t$ is usually unknown and needs to be estimated. Instead, $\Sn(\X_n,t_n)$ is used, where $t_n = t_n(\X_n)$ is a suitable location estimator. Often, 
(\ref{num:I}) and (\ref{num:II}) are mentioned as a theoretical justification for the use of $\Sn(\X_n,t_n)$, accompanied by a more or less explicit remark that $\Sn(\X_n,t)$ and $\Sn(\X_n,t_n)$ possess the same asymptotic behavior as long as $t_n$ is consistent for $t$ in some suitable sense. 
The purpose of this article is to close this gap and rigorously prove that, under weak conditions on $F$ and $t_n$, assertions (\ref{num:I}) and (\ref{num:II}) still hold true  if $t$ is replaced $t_n$.
We do not study the asymptotic variance and efficiency of the SSCM in this article, but refer the reader to \citet{Magyar2013} and \citet{Duerre2014}. Roughly speaking, the \SSCM\ achieves the same asymptotic efficiency as Tyler's estimator at spherical distributions, which is $1+ 2/p$ relative to the (suitably scaled) sample covariance matrix, but its asymptotic variance may get arbitrarily large for heteroscedastic data.  

The canonical location estimator for the \SSCM\ is the \emph{spatial median} 
\[
	\mu_n = \mu_n(\X_n) = \argmin_{\mbox{\scriptsize $\mu \in \R^p$}} \sum_{i=1}^n | X_i - \mu |.
\]
If the data points do not lie on a straight line and none of them coincides with $\mu_n$, the spatial signs with respect to the spatial median are centered \citep[][p.~228]{Kemperman1987}, i.e.\ $\sum_{i=1}^n s(X_i - \mu_n)  = 0$. Thus the \SSCM\ $\Sn (\X_n, \mu_n)$ is indeed the sample covariance matrix of the spatial signs of the centered observations, if the latter are centered by the spatial median. 
The theoretical counterpart, the \emph{spatial median of the distribution $F$} is 
\[
	\mu(F) = \argmin_{\mbox{\scriptsize $\mu \in \R^p$}} E \left( |X - \mu| -|X|\right).
\]
The spatial median always exists, and, if $F$ is not concentrated on a straight line, it is unique. 
For further details see, e.g., \citet{Kemperman1987}, \citet{Milasevic1987} and \citet{Koltchinskii2000}. If the first moments of $F$ are finite, then the spatial median allows the more descriptive characterization as the minimizing point of $E |X - \mu|$. The spatial median falls within the class of $M$-estimators, for which an elaborate asymptotic theory exists \citep[][Sec.~6]{Huber2009}. See \citet{magyar:2011} for the asymptotic distribution and finite sample efficiencies. The computation of the spatial median is also a thoroughly studied problem \citep[e.g.][]{weiszfeld:2009, gower:1974, vardi:2001}. Alternative names 
are $L_1$ median, mediancentre and space median. For a recent review see \citet{oja:2010}.

An alternative definition of the sample spatial sign covariance matrix with unknown location is
\[
		\Snstar(\X_n,t_n)=\frac{1}{n^*}\sum_{i=1}^n s(X_i-t_n) s(X_i-t_n)^T
\]
with $n^* = \# \{ 1 \le i \le n \mid X_i \neq t_n \}$. This definition is preferable for practical purposes. For instance, the trace of $\Snstar(\X_n,t_n)$ is always 1 (except for the uninteresting case that all observations $X_i$ coincide with $t_n$). The estimator $\Sn(\X_n,t)$, on the other hand, also contains the information how many of the data points $X_i$ coincide with $t$. This information is generally of little interest. More importantly, when $n^* < n$ is observed, it is in most cases not due to properties of the underlying distribution (about which one seeks to draw inference), but is an artefact of the measurement or the location estimator. When $X_i = t_n$, it is usually either caused by a coarse rounding of the data or an estimator $t_n$ that attains one of the data points with positive probability, which is common for many robust estimators, among them the spatial median. 
All asymptotic results of Section \ref{sec:asym} concerning the estimator $\Sn(\X_n,t_n)$ are also true for $\Snstar(\X_n,t_n)$. The difference between both estimators is, except for very small $n$, negligible for continuous population distributions. 

%
%
%
%
A referee raised the question of the general importance of the location estimation problem and pointed out that it can be elegantly circumvented by means of symmetrization. For scatter estimation in general (univariate or multivariate), the need for a prior location estimate is usually regarded as a kind of nuisance, making thorough derivations more involved than, e.g., for the problem of location estimation alone. (We are addressing explicitly this nuisance here.) One way to avoid this problem is to use symmetrized scatter estimators. Any scatter estimator
gives rise to a symmetrized estimator, which is simply the estimate computed from the pairwise differences, instead from the (suitably centered) data. 
To name two well-known univariate examples, the symmetrized version of the mean absolute deviation is known as Gini's mean difference (simply the mean of all pairwise distances). The $Q_n$ scale estimator proposed by \citet{RousseeuwCroux1993}, which is the lower sample quartile of all pairwise distances, can be regarded as a symmetrized version of the median absolute deviation \citep[MAD,][]{Hampel1974}. Roughly speaking, symmetrized estimators tend to be more efficient at the normal model, but less efficient at very heavy tailed distributions and less robust than the original estimator.
Symmetrization is often successfully applied to highly robust estimators to increase their efficiency while retaining a satisfactory degree of robustness. 

Multivariate symmetrized scatter estimators have been considered, e.g., in \citet{Duembgen1998} and \citet*{Sirkia2007}. They play an important role in robust principal component analysis, since they possess the so-called independence property: symmetrized scatter functionals are always diagonal matrices at multivariate distributions with independent margins. See e.g. \citet{Oja2006} or \citet{Tyler2009} for further details.
Particularly the symmetrized version of the \SSCM\ (which we simply want to call SSSCM here) has also been considered in \citet{Visuri2000}. It is denoted by $\rm TCM_2$ there, and the authors call it the \emph{spatial Kendall's $\tau$ covariance matrix}. It is indeed much more efficient than the SSCM at the normal distribution. For instance, at a bivariate spherical normal distribution, where the SSCM and the SSSCM both are consistent for $I_2/2$, the SSCM has an asymptotic relative efficiency (with respect to $\hat\Sigma_n / \trace(\hat\Sigma_n)$) of 50\%, whereas the SSSCM achieves 91\%, which can be deduced from results by \citet{Sirkia2009}.
However, despite their appealing efficiency properties, symmetrized estimators also have a few drawbacks, mainly a higher computational effort, and a less tractable asymptotic variance, making inferential procedures more laborious. The \SSCM, with its simplicity as a major appeal, will certainly be of continuing relevance in statistics.  

The rest of the paper is organized as follows: In Section \ref{sec:asym}, strong and weak consistency of the \SSCM\ are proven under the assumption of strong and weak consistency, respectively, of the location estimator $t_n$ and a moment condition on $|X-t|^{-1}$. Furthermore, the asymptotic normality of $ \Sn(\X_n,t_n)$ is shown, provided $t_n$ converges at the usual $\sqrt{n}$ rate. Section \ref{sec:sim} contains the results of a small simulation study demonstrating the effects of different location estimators and exploring the sensitivity of the convergence of the \SSCM\ on the aforementioned inverse moment condition. All proofs are deferred to the appendix.

\section{Asymptotic theory}
\label{sec:asym}
%

Our first result states conditions on the location estimator $t_n$ and the population distribution $F$ that guarantee weak and strong consistency, respectively, of the spatial sign covariance matrix with unknown location. 
\begin{theorem} \label{th:cons}
If there is an $\alpha \geq 0$ such that 
\begin{enumerate}[(I)]
\item \label{num:cons:I} 
$E |X-t|^{-1/(1+\alpha)}  < \infty$, \ and  
\item \label{num:cons:II}
there is a sequence $(t_n)_{n \in \N}$ of random $p$-vectors that satisfies 
\[
		n^\alpha |t_n - t| \leq T_n  \quad \mbox{a.s.},   
\] 
where $(T_n)_{n \in \N}$ is a sequence of random variables that converges almost surely (if $\alpha > 0$) or converges almost surely to 0 (if $\alpha = 0$), 
\end{enumerate}
then $\Sn(\X_n,t_n) \asc S(F,t)$. If $n^\alpha |t_n - t|$ is bounded in probability (for $\alpha > 0$) or converges in probability to 0 (for $\alpha = 0$), then $\Sn(\X_n,t_n) \cip S(F,t)$.
\end{theorem}
We have the following remarks concerning Theorem \ref{th:cons}:
\begin{enumerate}[(I)] 
\item \label{rem:1:I}
The primary assertion of Theorem \ref{th:cons} is the case $\alpha = 0$. If $t_n$ is strongly or weakly consistent for $t$, then $\Sn(\X_n,t_n)$ is strongly or weakly consistent, respectively, for $S(F,t)$ if the first moment of $1/|X-t|$ is finite. Theorem \ref{th:cons} states further that, if there is information on the rate of convergence of $t_n$, i.e., (\ref{num:cons:II}) is fulfilled for some $\alpha > 0$, the assumption on $F$ can be weakened, requiring less than first moments of $1/|X-t|$. Note that condition (\ref{num:cons:I}) of Theorem \ref{th:cons} gets weaker with increasing $\alpha$, whereas (\ref{num:cons:II}) gets stronger. 

\item \label{rem:1:II}
For any reasonable location estimate, we expect $\sqrt{n}(t_n - t)$ to converge in distribution. Then we have weak consistency of $\Sn(X_n,t_n)$ if 
$E |X-t|^{-2/3} < \infty$.
\item \label{rem:1:III}	
The situation is slightly different for strong consistency. If we take the mean as location estimator, i.e.\ $t_n = n^{-1}\sum_{i=1}^n X_i$ and $t = E(X)$, we know by the law of the iterated logarithm that $\sqrt{n}|t_n - t|$ does not converge almost surely, but $n^{\alpha} |t_n - t| \asc 0$ for any $\alpha < 1/2$. Thus 
$E |X - t|^{-2/3-\delta} < \infty$ for some $\delta > 0$ is required for strong consistency.

\item \label{rem:1:IV}
Assumption (\ref{num:cons:I}) of Theorem \ref{th:cons} requires that the probability mass is not too strongly concentrated in the vicinity of $t$. This seems intuitive: for many observations $X_i$ being very close to $t$, the spatial signs $s(X_i - t_n)$ and $s(X_i - t)$ will vastly differ, even if $t_n$ is close to $t$. In this sense, assumption (\ref{num:cons:I}) accounts for the discontinuity of the sign function $s(\cdot)$ at 0. However, it is a very mild condition, it is fulfilled, e.g., if the density of $F$ is bounded at $t$.

\item \label{rem:1:V}
A continuous, elliptical distribution $F$ is characterized by a density $f$ of the form 
\be \label{eq:dens}
		f(x) = \det(V)^{-\frac{1}{2}} g\big\{(x-\mu)^T V^{-1} (x-\mu)\big\}
\ee
for a function $g:[0,\infty) \to [0,\infty)$, a positive definite $p\times p$ matrix $V$ and a $p$-vector $\mu$. The parameter $\mu$ coincides with the spatial median of $F$ as well as the mean, provided the latter exists. Then $E |X-\mu|^{-1/(1+\alpha)}  < \infty$ is fulfilled if 
\[
	g(z) = O\big( z^{\{ 1/(1+\alpha)-p\}/2+\delta  }  \big)   \qquad \mbox{as} \ \ z \to 0
\]
for some $\delta > 0$. In particular, it is always fulfilled if $g$ is bounded at the origin, thus for instance for all normal and elliptical $t$ distributions in any dimension. Note that the boundedness is not a necessary condition. 
\end{enumerate}

%
%
%
%
%
%
%
%
%
%
%
The next result gathers conditions that ensure the asymptotic normality of the SSCM with unknown location.

\begin{theorem} \label{th:an1}
If
\begin{enumerate}[(I)]
\item \label{num:an1:I}
	$ E|t_n-t|^4 = O(n^{-2})$,
\item \label{num:an1:II} 
	$E |X-t|^{-3/2} <\infty$,
\item \label{num:an1:III}
	$E\left\{ \frac{X-t}{|X-t|^2} \right\} =0$ \ and
	$E\left\{ \frac{ (X-t)^{(i)} (X-t)^{(j)} (X-t)^{(k)} }{|X-t|^4 } \right\} =0$ \ for\  $i,j,k=1,\ldots,p$,
\end{enumerate}
 where $(X-t)^{(i)}$ denotes the $i$ \!th component of the random vector $X-t$, then
\[
	\sqrt{n} \vec\, \{ \Sn(\X_n,t_n)-S(F,t) \}   \cid N\!_{p^2}(0, W),
\] 
with $W$ being defined in Section \ref{sec:intro}.
\end{theorem}
Assumption (\ref{num:an1:III}) of Theorem \ref{th:an1} imposes some form of symmetry of the distribution $F$ around the point $t$. It is fulfilled, e.g., if $(X-t) \sim -(X-t)$, thus in particular for elliptical distributions with center $t$.
If $F$ is not symmetric around $t$, asymptotic normality can nevertheless be shown, but we require $\Sn(\X_n,t)$ and $t_n$ to converge jointly, and the asymptotic covariance matrix of $\Sn(\X_n,t_n)$ does in general not coincide with $W$. This is the situation of our last theorem. 
\begin{theorem} \label{th:an2}
If 
\begin{enumerate}[(I)]
\item \label{num:an2:0}
	$ E|t_n-t|^4 = O(n^{-2})$,
\item \label{num:an2:II}
$E|X-t|^{-3/2} < \infty$, and
\item \label{num:an2:I}
$ \displaystyle 
\sqrt{n}
\begin{pmatrix}
t_n - t\\
\vec\, \{ \Sn(\X_n,t) - S(F,t)\}
\end{pmatrix} \cid  N\!_{p + p^2} (0, \, \Xi)$ \ \ 
for a symmetric $p(p+1) \times p(p+1)$ matrix $\Xi$,
\end{enumerate}
then 
\[
		\sqrt{n} \vec \, \{ \Sn(\X_n,t_n) - S(F,t) \}   \cid  N\!_{p^2}(0, A \Xi A^T)
\]
with $A = (B , I_{p^2}) \, \in \R^{p^2\times p(p+1)}$, where $I_{p^2}$ is the $p^2$-dimensional identity matrix and $B \in \R^{p^2\times p}$ is given by
\[
	B \ = \ 2 E \left[ \frac{\{(X-t) \otimes (X-t)\} (X-t)^T }{|X-t|^4} \right]
	\, - \, E \left\{ \frac{X-t}{|X-t|^2}\right\} \otimes I_p \, - \, I_p \otimes E \left\{ \frac{X-t}{|X-t|^2}\right\}.
\]
\end{theorem}
Again, the assumptions are rather weak. If the fixed-location SSCM and the location estimator converge individually, the joint convergence (\ref{num:an2:I}) is usually also fulfilled, although a thorough proof may be tedious. 
For the sample mean, it can be seen fairly easily, and hence Theorem \ref{th:an2} implies that, e.g., the SSCM with sample mean location is asymptotically normal for any distribution with finite second moments and bounded density. 

We have two closing remarks:
\begin{enumerate}[(I)] 
\item 
In all convergence results in this section, the population distribution $F$ is completely arbitrary except for the moment condition on $1/|X - t|$, in particular no continuity is required. If we assume $F$ to be continuous, condition (\ref{num:an1:I}) of Theorem \ref{th:an1} can be weakened to the boundedness in probability of $\sqrt{n}(t_n-t)$, and in Theorem \ref{th:an2} it can be dropped altogether. The connection is given by Lemma \ref{lem:1} in the appendix. However, the continuity of $F$ is not an essential assumption for these asymptotic statements to hold, and it can, e.g., be exchanged for uniform boundedness of the fourth moments of $\sqrt{n}(t_n-t)$, which is, while being far from strict, a condition often fulfilled and easy to verify. 

\item
The results of this section should be compared to the analogous ones for Tyler's shape matrix \citep[][Sec.~4]{Tyler1987}. The conditions on the population distribution $F$ and the location estimate $t_n$ are similar. 
\end{enumerate}

\section{Simulations}
\label{sec:sim}

The simulations section has two parts: in Subsection \ref{subsec:vs}, we want to get an impression how well, for several $n$ and different population distributions, the distribution of $\sqrt{n} \{ \Sn(\X_n,t_n) - S(F,t) \}$ is approximated by its Gaussian limit. We study in particular the differences when using the mean and the spatial median as location estimator. In Subsection \ref{subsec:singularity}, we consider distributions with unbounded densities to examine what happens near the limit case of the inverse moment condition of Theorem \ref{th:cons}.
\subsection{Spatial median versus mean}
\label{subsec:vs}
We sample from a bivariate, centered normal distribution with covariance matrix 
\be \label{eq:Sigma}
	\Sigma = 
	\begin{pmatrix}
		1 & 1/2 \\
	  1/2 & 1 \\
	\end{pmatrix}
\ee	
and from a bivariate, centered, elliptical $t$-distribution with two degrees of freedom and the same shape matrix $\Sigma$. The corresponding spatial sign covariance matrix is 
\[
	S(F,0) = 
	\begin{pmatrix}
		1/2 & 0.13397 \\
	  0.13397 & 1/2 \\
	\end{pmatrix},
\]	
cf.~\citet{Duerre2014}. With the variance of the $t_2$ distribution not being finite, we have that, when using the mean as location estimator, the sequence $\sqrt{n} (t_n-t)$ does not converge. Hence we expect the SSCM not to be asymptotically normal in this case. 

 
For both distributions and for various sample sizes, we compute the SSCM with mean and spatial median as location estimator. For each setting, 100,000 samples are evaluated, which yields a fairly precise Monte-Carlo approximation of the actual distribution of the SSCM.
\begin{figure}[t]
	\centering
		\includegraphics[width=\textwidth]{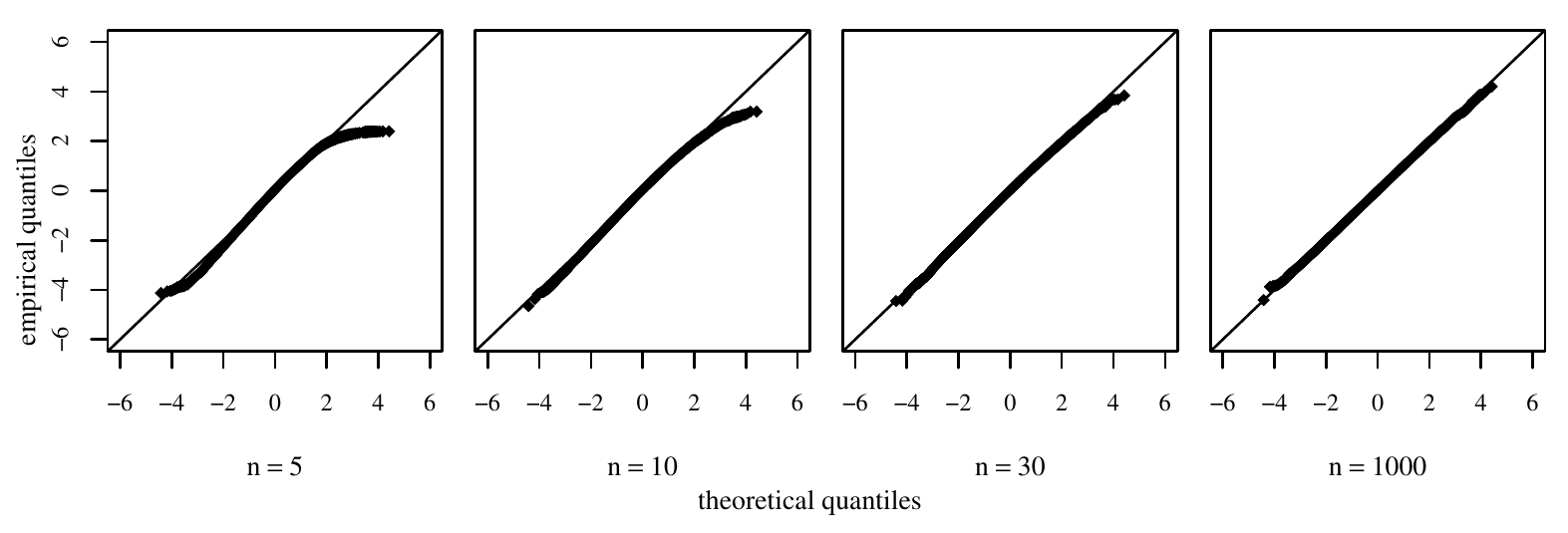}
		\includegraphics[width=\textwidth]{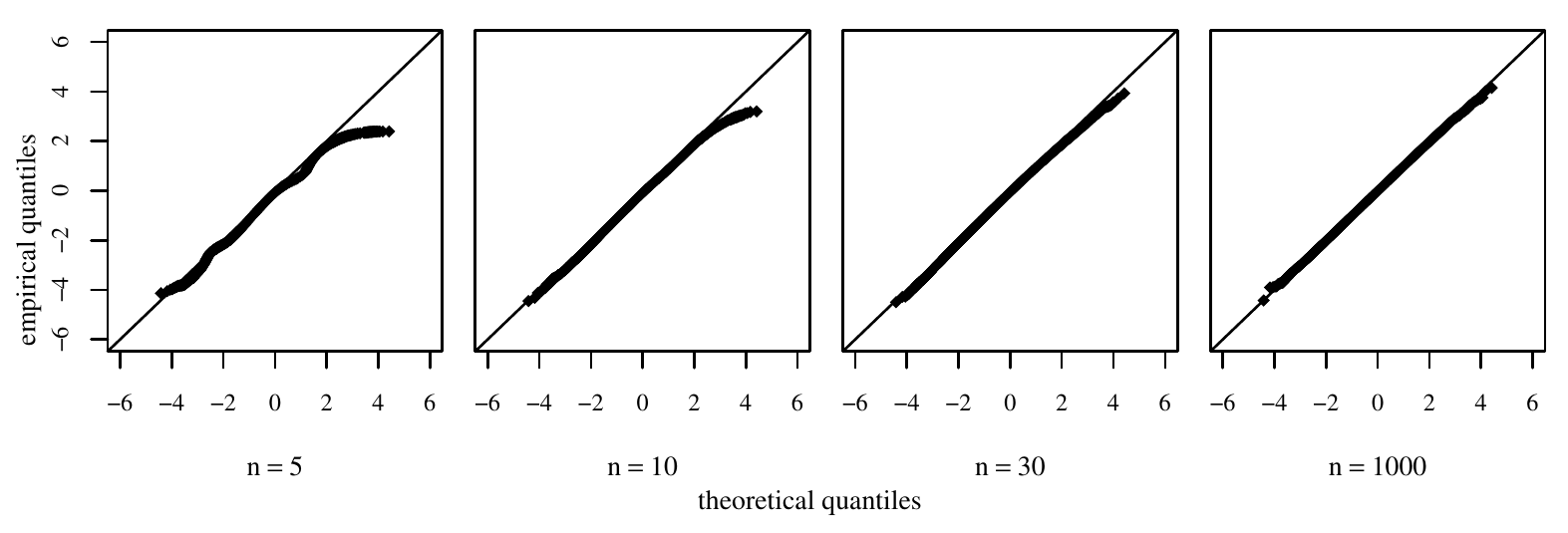}
	\caption{QQ plots for the off-diagonal element $[\!\sqrt{n} \{ \Sn(\X_n,t_n) - S(F,t)\}]_{1,2}$ at the normal distribution with covariance (\ref{eq:Sigma}), when using mean (top row) and spatial median (bottom row) as location estimator.}
	\label{fig:1}
\end{figure}
In Figure \ref{fig:1} we see QQ plots of the off-diagonal element 
$\left[ \sqrt{n} \{ \Sn(\X_n,t_n) - S(F,t) \} \right]_{1,2}$ for the normal population distribution and $n = 5, 10, 30, 1000$. The simulated distribution is plotted against the asymptotic distribution, which is also given in \citet{Duerre2014}. The pictures in the first row (mean) and second row (spatial median) are very similar, and we note very little difference for the two location measures under normality. We suspect that the bumpiness of the QQ plot for the spatial median for $n=5$ (lower left corner) is due to the spatial median's tendency to coincide with one of the data points.
For both estimators, we find the distribution of $\left[ \sqrt{n} \{ \Sn(\X_n,t_n) - S(F,t) \} \right]_{1,2}$ to have light tails for $n =  5$ (left column). The reason is the bounded range $[-0.5,0.5]$ of the estimate $\{\Sn(\X_n,t_n)\}_{1,2}$. If $n$ is very small, this boundedness is also visible in the QQ plots of the centered and $\sqrt{n}$-scaled estimator. The asymmetry (lighter upper tail) is due to the true value 0.134 not lying in the center of the possible range. 
However, this departure from normality quickly vanishes as $n$ increases, and the normal limit generally provides a very good approximation for fairly moderate sample sizes.
%
%


\begin{figure}[t]
\centering
	\includegraphics[width=1\textwidth]{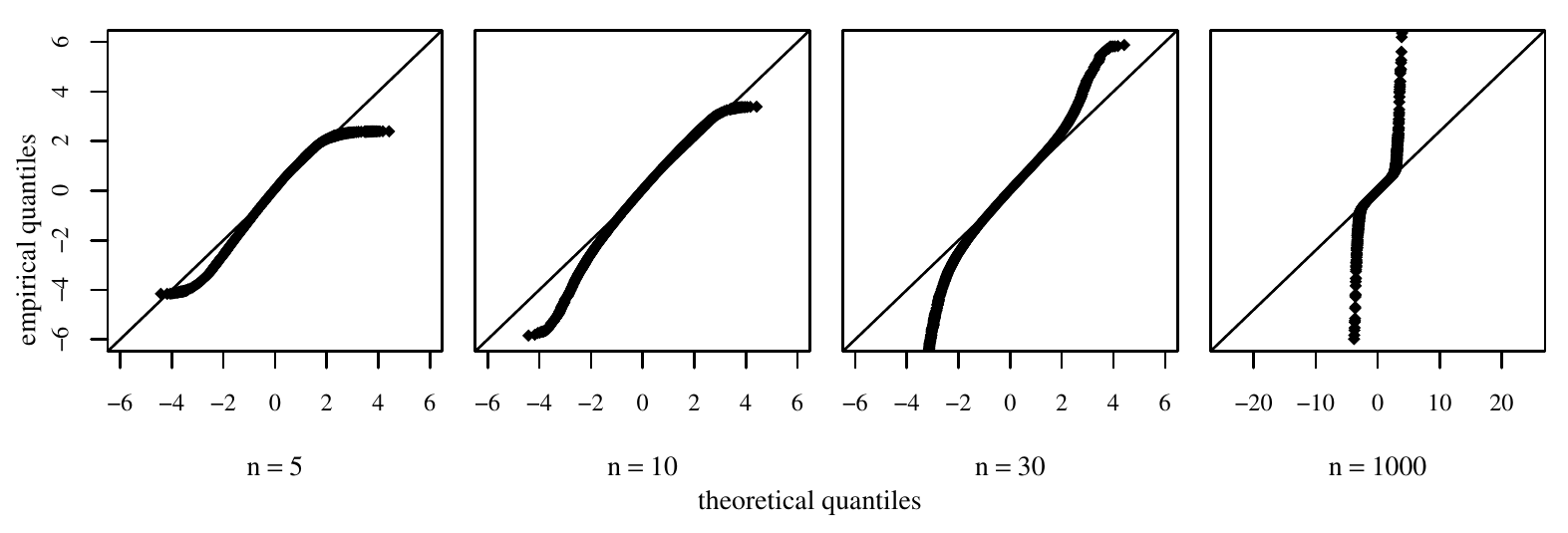}
	\includegraphics[width=1\textwidth]{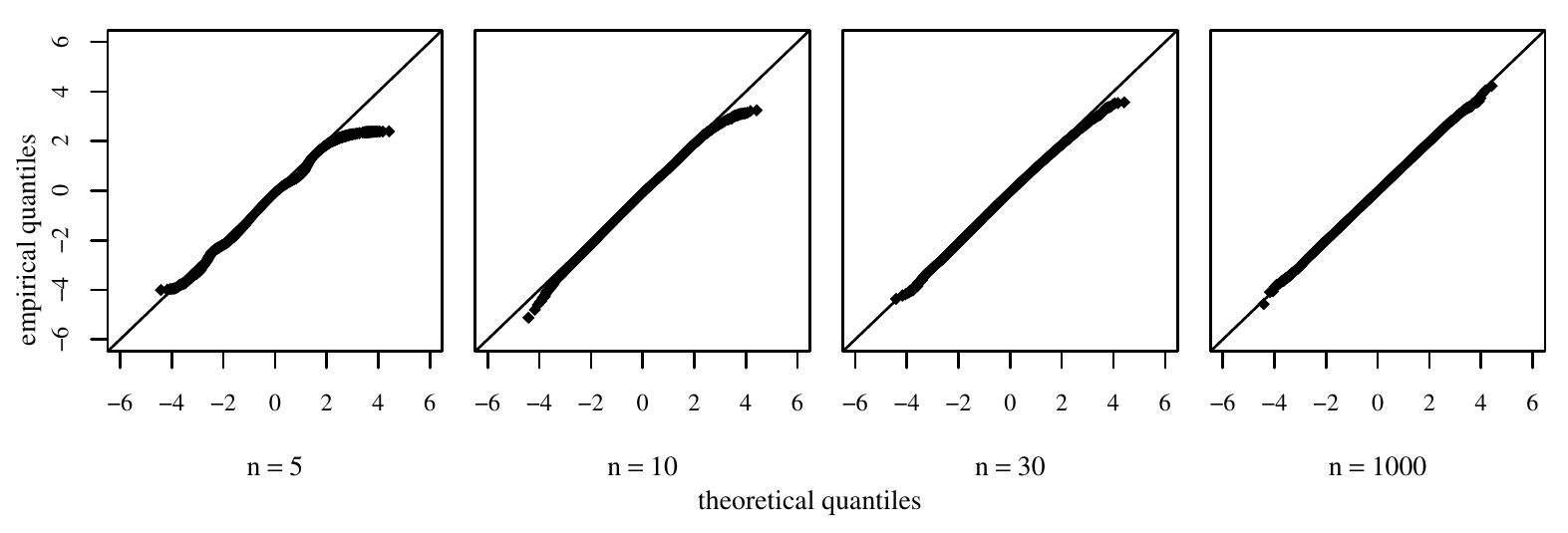}
\caption{QQ plots for $[\!\sqrt{n} \{ \Sn(\X_n,t_n) - S(F,t)\}]_{1,2}$ at $t_2$ distribution with shape matrix (\ref{eq:Sigma}), when using mean (top row) and spatial median (bottom row) as location estimator.}
\label{fig:2}
\end{figure}
Figure \ref{fig:2} shows analogous results for the $t_2$ distribution. Here we observe a qualitatively different behavior of the SSCM depending on the location estimate. In case of the mean (first row), the off-diagonal element $[\!\sqrt{n} \{ \Sn(\X_n,t_n) - S(F,t)\}]_{1,2}$ is clearly non-normal with heavy tails, which persist also for large $n$. As for the spatial median, the results for normal and $t_2$ distribution are very similar: the second row of Figure \ref{fig:2} is almost identical to the second row of Figure \ref{fig:1}.

Additional to the small bivariate example, we present some simulation results for higher dimensions. The results for various $n$ and $p$ are summarized in Tables \ref{tab:1} (normal distribution) and \ref{tab:2} (for the $t_2$ distribution). The set-up is slightly different: the true covariance is $I_p$, corresponding to $S(F,0) = p^{-1} I_p$, and we do not concentrate on a single element, but consider the $L_2$ distance of the matrix estimate to the true $S(F,0)$. The averages of $n|| \Sn(\X_n,t_n) - p^{-1} I_p||^2$ (based on 10,000 runs in each setting) are given. Note that the matrix distance does not blow up as $p$ increases. But this is not surprising considering that $||\Sn(\X_n,t_n)||^2$ is bounded by 1 in any dimension $p$. 
Furthermore, the \SSCM\ with known location is included in the comparison. For instance, at Table \ref{tab:1} we observe observe that the loss for not knowing the location at is about 25\% for $n = 5$  and about 3\% for $n = 30$. 
In contrast, it appears to make practically no difference, in case of normality, which location estimator is chosen. This is in line with the observation that the mean and the spatial median have a similar efficiency at high-dimensional spherical normal distributions, cf.\ e.g.\ \citet{magyar:2011}.

We have also carried out simulations for a variety of other elliptical distributions, with different generators $g$ and different shape matrices $V$. 
The general picture throughout is the same that is conveyed by the examples here: choosing mean or spatial median as location estimate makes little difference in situations where both estimators are root-$n$-consistent. (The differences are more pronounced for ``strongly shaped'' matrices $V$, where the spatial median is relatively less efficient.) For distributions without finite second moments, taking the sample mean substantially impairs the sample SSCM.

\begin{table}[t]
\begin{tabular}{l|ccc|ccc|ccc|ccc}
\multirow{2}{*}{\diagbox{p}{n}}&\multicolumn{3}{c|}{5}&\multicolumn{3}{c|}{10}&\multicolumn{3}{c|}{30}&\multicolumn{3}{c}{1000}\\
 & kn & mean & med & kn & mean & med & kn & mean & med & kn & mean & med\\
\hline
10	 & 0.901 & 1.090 & 1.074 & 0.900 & 0.983 & 0.979 & 0.899 & 0.924 & 0.923 & 0.897 & 0.897 & 0.897\\
50	 & 0.981 & 1.217 & 1.213 & 0.980 & 1.085 & 1.084 & 0.979 & 1.012 & 1.012 & 0.980 & 0.981 & 0.981\\
200	 & 0.995 & 1.241 & 1.241 & 0.995 & 1.104 & 1.104 & 0.995 & 1.029 & 1.029 & 0.995 & 0.996 & 0.996\\
1000 & 0.999 & 1.248 & 1.248 & 0.999 & 1.110 & 1.110 & 0.999 & 1.033 & 1.033 & 0.999 & 1.000 & 1.000\\
\end{tabular}
\caption{Normal distribution; average of $n ||\Sn - p^{-1} I_p||^2$ (10,000 runs) for different samples $n$, dimensions $p$ when the location is known (kn) or estimated by the mean or the spatial median (med).}
\label{tab:1}
\end{table}
\begin{table}[t]
\begin{tabular}{l|ccc|ccc|ccc|ccc}
\multirow{2}{*}{\diagbox{p}{n}}&\multicolumn{3}{c|}{5}&\multicolumn{3}{c|}{10}&\multicolumn{3}{c|}{30}&\multicolumn{3}{c}{1000}\\
 & kn & mean & med & kn & mean & med & kn & mean & med & kn & mean & med\\
\hline
10	 & 0.900 & 1.460 & 1.084 & 0.896 & 1.524 & 0.984 & 0.899 & 1.629 & 0.926 & 0.897 & 1.576 & 0.898\\
50	 & 0.981 & 1.649 & 1.222 & 0.980 & 1.711 & 1.090 & 0.980 & 1.889 & 1.015 & 0.980 & 1.561 & 0.981\\
200	 & 0.995 & 1.700 & 1.249 & 0.995 & 1.772 & 1.110 & 0.995 & 1.923 & 1.031 & 0.995 & 1.858 & 0.996\\
1000 & 0.999 & 1.705 & 1.256 & 0.999 & 1.753 & 1.116 & 0.999 & 1.930 & 1.036 & 0.999 & 1.662 & 1.000\\
\end{tabular}
\caption{$t_2$ distribution; average of $n ||\Sn - p^{-1} I_p||^2$ (10,000 runs) for different $n$ and $p$ when the location is known (kn) or estimated by the mean or the spatial median (med).}
\label{tab:2}
\end{table}

%
%
%
%
%
%

\subsection{Singularity distribution}
\label{subsec:singularity}
The second goal of our small simulation study is to assess the sensitivity of the convergence of $\Sn(\X_n,t_n)$ with respect to condition (\ref{num:cons:I}) of Theorem \ref{th:cons}, i.e., 
the probability mass concentration around $t$. For that purpose, we consider the $p$-variate elliptical distribution $F_{\gamma,p}$, $\gamma > 0 $, with density $f_{\gamma,p}$ given by 
$\mu = 0$, 
$V = I_p$ and 
\[
	g_{\gamma,p}(z) = c_{\gamma,p} z^{\gamma-p/2} \varind{[0,1]}(z), 
\] 
cf.~(\ref{eq:dens}).
The factor $c_{\gamma,p} = \gamma \Gamma(p/2)/\pi^{p/2}$ scales the corresponding $p$-variate density $f_{\gamma,p}$ to 1. 
If $X \sim F_{\gamma,p}$, the norm $|X|$ has, for any $p$, the density
\[
	h_\gamma(z) = 2 \gamma z^{2\gamma - 1} \varind{[0,1]}(z).
\]
The densities $f_{\gamma,p}$ and $h_\gamma$ are well defined for any $\gamma > 0$. A smaller value of $\gamma$ corresponds to a stronger singularity at the origin, and $\gamma = 0$ constitutes the limit case, in which the densities are not integrable. 
Furthermore, if $-\beta/2 < \gamma$, we have $E|X|^{\beta} < \infty$. Thus, using the mean as location estimator, Theorem \ref{th:cons} grants strong consistency for $\gamma > 1/3$.

In the simulations we consider $\gamma = 0.05, 0.1, 0.15, \ldots, 0.45$ and several sample sizes $n$ ranging from 10 to 20000. 
Figures \ref{fig:3} 
illustrates the $L_1$ consistency of the SSCM depending on $\gamma$ for $p = 2$ and $p = 10$: the absolute error $\varepsilon$ of an off-diagonal element of the SSCM is plotted against $\gamma$ and $n$ with the mean (left) and the spatial median (right) as location estimator. Each grid point is the average of 10,000 repetitions. The smaller errors for $p=10$ (bottom row) are due to the fact that we examine one singly entry of the SSCM, the magnitude of which is of order $p^{-2}$ as $p$ increases.  

For the mean, we observe a decline of the absolute error with increasing $n$ for all $\gamma$ considered, but the convergence is very slow for small values of $\gamma$. Also here, the SSCM with the spatial median behaves qualitatively different: the decay of the absolute error appears to be equally fast for all $\gamma$. This is a plausible observation. For the univariate median, the asymptotic efficiency is proportional to the squared density at the true median, and the boundedness of the density at the median is necessary to prove asymptotic normality. In case of a singularity, the univariate sample median is known to converge faster than at the usual $\!\sqrt{n}$ rate. The situation is similar for the spatial median: the efficiency increases with higher probability mass concentration at the true spatial median $\mu$, and the boundedness of the density is also a standard assumption for the asymptotic normality of the sample spatial median \citep[see e.g.][]{mottonen:nordhausen:oja:2010,magyar:2011}.
The sensitivity of the SSCM $\Sn(\X_n,t_n)$ to a singularity at $t$, due to discontinuity of the spatial sign, and the increased efficiency of the spatial median in such a situation are opposing effects that seem to nullify each other. A thorough theoretical investigation of this situation seems to be an open problem.  

\begin{figure}[ht]
	\centering
		\includegraphics[width=0.45\textwidth]{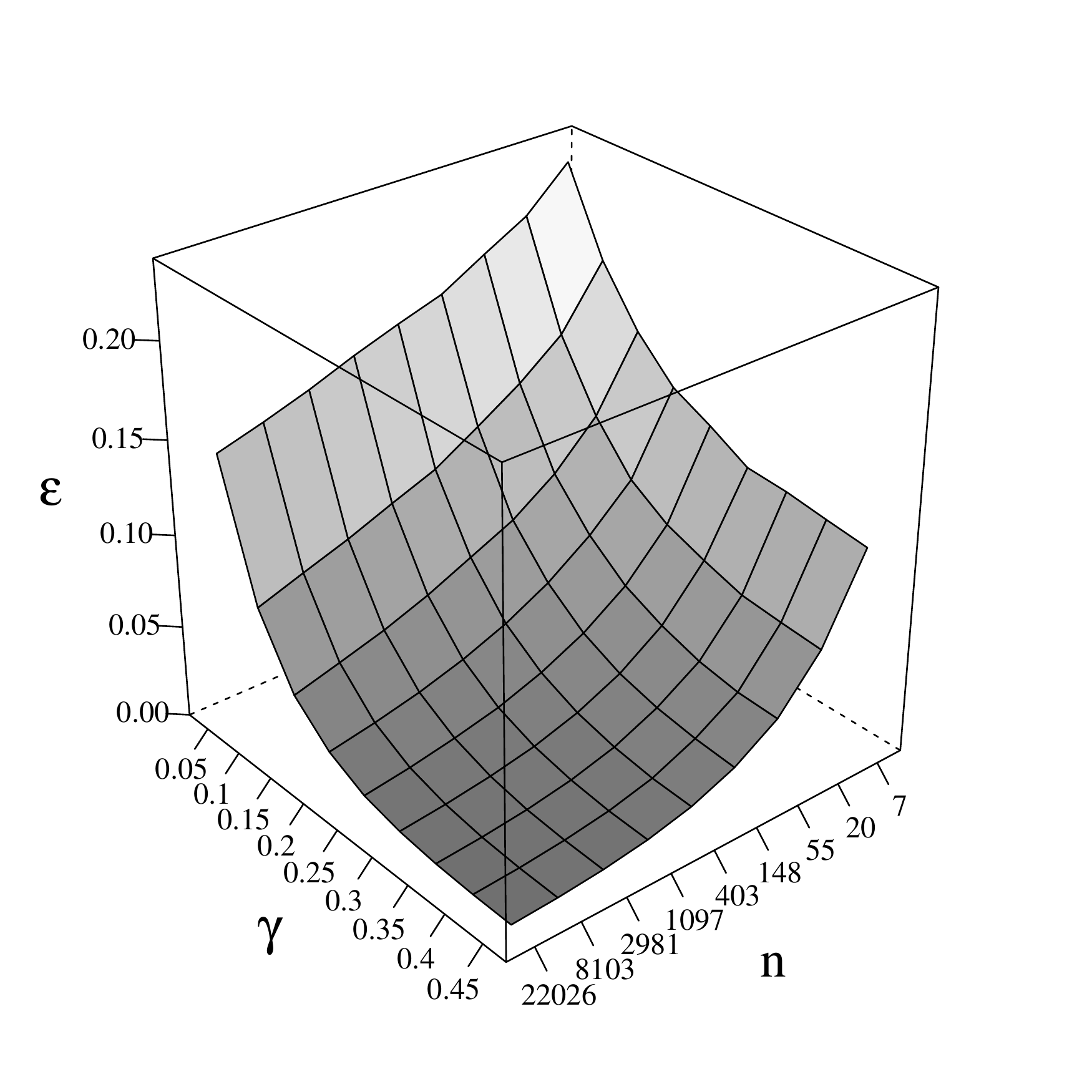}
		\includegraphics[width=0.45\textwidth]{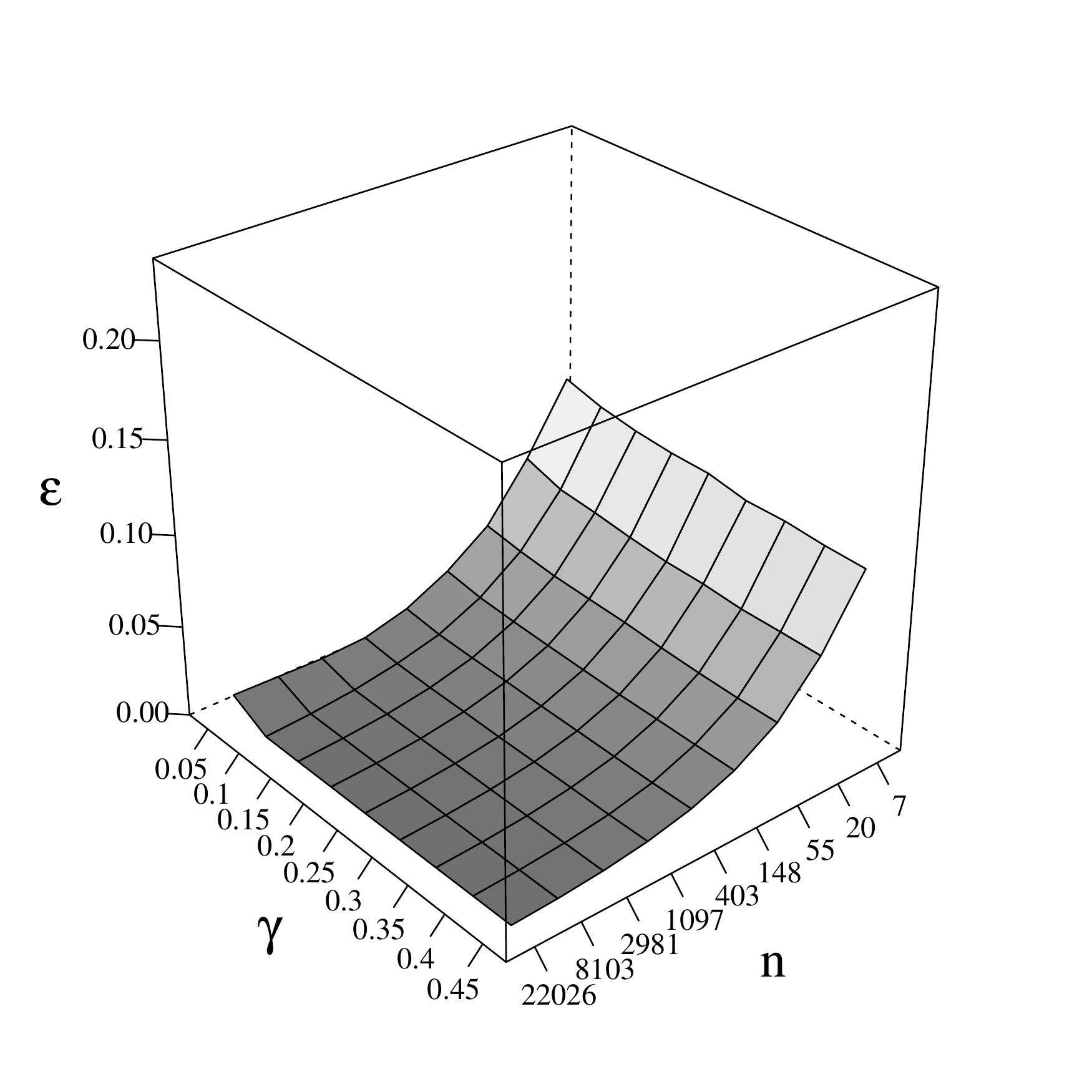}

		\vspace{-3.ex}
		\includegraphics[width=0.45\textwidth]{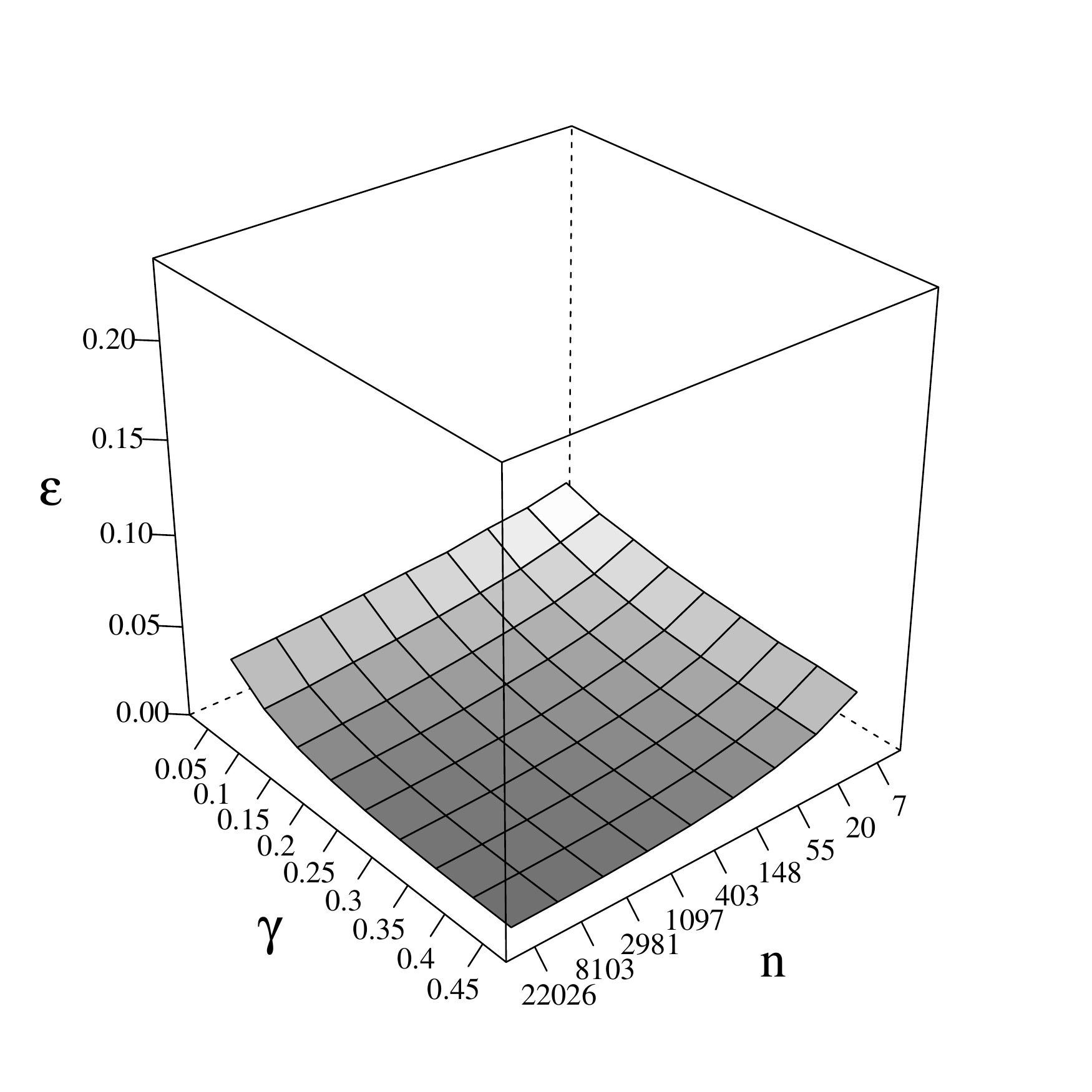}
		\includegraphics[width=0.45\textwidth]{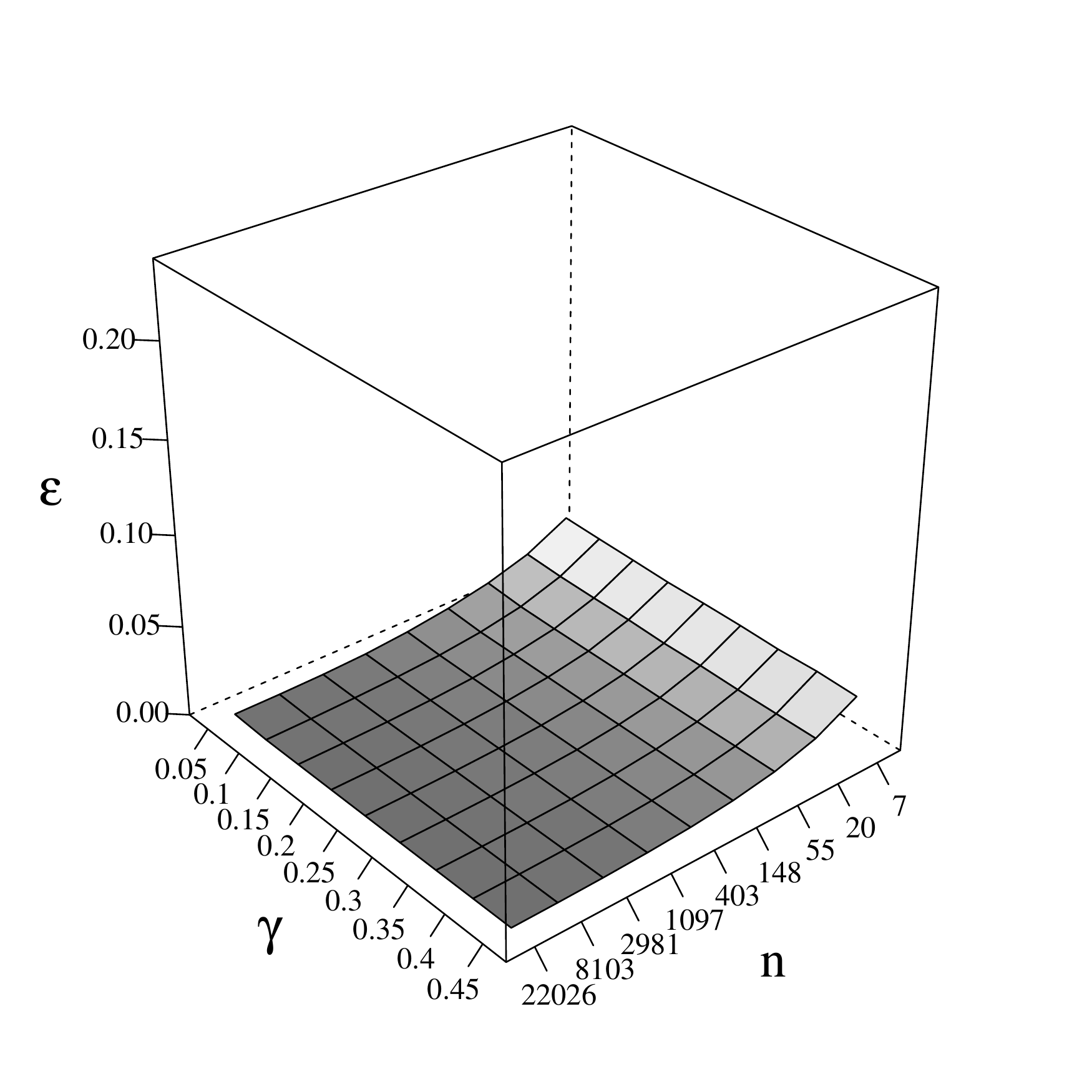}
		
		\vspace{-2.ex}
	\caption{Mean absolute error $\varepsilon$ of the off-diagonal element of the SSCM at distribution $F_{\gamma,p}$ for $p = 2$ (top row) and $p = 10$ (bottom row) and different values of $\gamma$ and $n$. Location estimated by mean (left) and spatial median (right).}
	\label{fig:3}
\end{figure}

In conclusion, we note that in all our simulations, the SSCM with the spatial median was non-inferior to the SSCM with the sample mean, with clear superiority in the case of the $t_2$ distribution and the singularity distribution $F_{\gamma}$. The use of the spatial median as location estimator is advisable and clearly preferable to the  mean. Besides the conceptual kinship of the SSCM and the spatial median and their common good robustness properties, the spatial median is also advantageous in the case of high mass concentration at the center.

All simulations were done in R 2.11.1 \citep{R} using the packages mvtnorm \citep{mvtnorm} for generating elliptical distributions and pcaPP \citep{pcaPP} for computing the spatial median.

\appendix
\section{Proofs}
In the following proofs, we let, without loss of generality, $t = 0$, and throughout write $\Sn(t_n)$ short for $\Sn(\X_n,t_n)$, $\Sn$ for $\Sn(\X_n,0)$ and $S$ for $S(F,0)$. The assumption $E|X|^{-1} < \infty$ is part of all three theorems, which implies $P(X=0) = 0$. Thus, we restrict our attention to the case $X_i \neq 0$ for all $i \in \N$, but note that, with some notational effort, we can generalize the results to distributions having an atom at the origin. 
We use $\|\cdot\|$ to denote the Frobenius norm of a real matrix, i.e., $\|A\| = \{ \trace(A^T A) \}^{1/2}$ for any $A \in \R^{m\times n}$. Letting $\vec A$ denote the $m n$-vector obtained by stacking the columns of $A \in \R^{m \times n}$ from left to right underneath each other, we have $\|A\| = |\vec A|$. 


\begin{proof}[Proof of Theorem \ref{th:cons}]
First we treat weak consistency. We show $\|\Sn(t_n)-\Sn\| \cip 0$ for $n\to \infty$. Associated with the random vector $t_n$ is the following random partition of the space $\R^p$, 
\[
	\calA_n = \left\{ \,  x \in \R^p \ \middle| \ \,  |x-t_n| \ge \textstyle \frac{1}{2} |x| \, \right\}, \qquad 
	\calA_n^C = \left\{ \,  x \in \R^p \ \middle| \ \,  |x-t_n| < \textstyle \frac{1}{2} |x| \, \right\},
\]
and the random partition 
$\calI_n =   \{  1 \le i \le n \, | \,  X_i \in \calA_n \}$, \ 
$\calI_n^C = \{  1 \le i \le n \,  | \,  X_i \in \calA_n^C \}$ 
of the index set $\{1,\ldots,n\}$.
Letting $\Gamma_i = s(X_i)s(X_i)^T$ and $\Gamma_i(t_n) = s(X_i-t_n)s(X_i-t_n)^T$, $i \in \N$, we have
\[
	\|\Sn(t_n)-\Sn\| 
	\le \frac{1}{n}\sum_{i=1}^n \| \Gamma_i(t_n) - \Gamma_i\| 
	= \frac{1}{n}\sum_{i \in \calI_n} \| \Gamma_i(t_n) - \Gamma_i\|
	    + \frac{1}{n}\sum_{i \in \calI_n^C} \| \Gamma_i(t_n) - \Gamma_i\|.
\] 
Call the first sum on the right-hand side $C_n$ and the second sum $D_n$. We show convergence to zero in probability for the random variables $C_n$ and $D_n$ separately, starting with $C_n$. For $X_i \neq 0$ and $X_i \neq t_n$ we have $\|\Gamma_i\| = \|\Gamma_i(t_n)\| = 1$ and 
\[
	\trace\left\{ \Gamma_i^T \Gamma_i(t_n) \right\} \ = \ 
	2\frac{|X_i|^2|t_n|^2-(X_i^Tt_n)^2}{|X_i-t_n|^2|X_i|^2} 
\]
and consequently
\[
	\| \Gamma_i(t_n) - \Gamma_i \| = 
	\frac{\left\{ 2 \left( |X_i|^2 |t_n|^2 - (X_i^T t_n )^2 \right) \right\}^{1/2}}{ |X_i| \, |X_i - t_n| },
\]
which, for $i \in \calI_n$, is bounded from above by $\sqrt{2} |t_n|/|X_i|$. Hence
\[
	C_n \le \frac{1}{n} \sum_{i \in \calI_n} \frac{\sqrt{2} |t_n|}{|X_i|}
	\le \frac{1}{n} \sqrt{2} |t_n| \sum_{i=1}^n |X_i|^{-1}
  = \sqrt{2}\, \bigg[ n^\alpha |t_n| \bigg] \, \left[ \frac{1}{n^{1+\alpha}} \sum_{i=1}^n |X_i|^{-1} \right].
\]
If $\alpha = 0$, the first factor in brackets converges in probability to zero by assumption (\ref{num:cons:II}) of Theorem \ref{th:cons}, and the second factor converges almost surely to $E|X|^{-1}$ by Kolmogorov's strong law of large numbers. If $\alpha > 0$, then the first factor is bounded in probability and the second factor converges to zero almost surely by Marcinkiewicz's strong law of large numbers \citep[e.g.][p.~255]{loeve:1977}.

It remains to show $D_n \cip 0$. Noting that always $\| \Gamma_i(t_n) - \Gamma_i\| \le 2$, we find that 
\be \label{eq:Dn}
	D_n/2 \le \frac{1}{n} \sum_{i=1}^n \varind{\calA_n^C}(X_i).
\ee
We prove that the right-hand side converges to zero in probability by showing its $L_1$ convergence. Let $\varepsilon > 0$. Since $E|X|^{-1} < \infty$, there is a $\delta$-ball $B_\delta = \{ x \in \R^p \mid |x| \le \delta\}$ around 0 with $F(B_\delta) < \varepsilon/2$, where $F$ denotes the distribution of $X$. Further, since $t_n \cip 0$, and $|t_n| \le \delta/2$ implies $\calA_n^C \subset B_\delta$, there is an $n_0 \in \N$ such that $P(\calA_n^C \subset B_\delta) \ge 1 - \varepsilon/2$ for all $n \le n_0$. Thus, for every fixed $i \in \N$ and all $n \le n_0$, we have
\begin{eqnarray*}
	P(X_i \in \calA_n^C) & = &  P(X_i \in \calA_n^C, \calA_n^C \subset B_\delta) \ + \ P(X_i \in \calA_n^C, \calA_n^C \nsubset B_\delta) \\
	& \le &  P(X_i \in B_\delta) \ + \ P(\calA_n^C \nsubset B_\delta) \ \le \varepsilon, 
\end{eqnarray*} 
and finally $ E \{ n^{-1} \sum_{i=1}^n \varind{\calA_n^C}(X_i) \} \le n^{-1} \sum_{i=1}^n P(X_i \in \calA_n^C) \le \varepsilon$ for all $n \ge n_0$. Thus, the right-hand side of (\ref{eq:Dn}) converges in $L_1$, and hence in probability, to zero, and so does $D_n$. The proof of weak consistency is complete. 

As for strong consistency, we treat $C_n$ completely analogously. For $D_n$, we show, as before, that the right-hand side of (\ref{eq:Dn}) converges to zero, but now almost surely. Let, again, $\varepsilon > 0$ be arbitrary. Further, let $Z_{n,i} = \varind{\calA_n^C}(X_i)$ and $\tZ_{\delta,i} = \varind{B_{\delta}}(X_i)$, where, as before, $B_\delta$ denotes the $\delta$-ball around 0. We chose $\delta > 0$ such that $F(B_\delta) < \varepsilon$. We use $(\cdot)_+$ to denote the non-negative part, i.e., $(x)_+ = x$ for $x > 0$ and $(x)_+ = 0$ for $x \le 0$. Then
\[
	\frac{1}{n} \sum_{i=1}^n Z_{n,i}
	\ \le \  \frac{1}{n} \sum_{i=1}^n \tZ_{\delta,i} 
						\ + \ \frac{1}{n} \sum_{i=1}^n \left(Z_{n,i} - \tZ_{\delta,i} \right)_+ 		 
\]
The first summand converges to $F(B_\delta)$, which is smaller than $\varepsilon$. For the second summand, we note that $|t_n| \le \delta/2$ implies $(Z_{n,i} - \tZ_{\delta,i})_+ = 0 $ for all $i \in \N$. Thus, since $|t_n| \asc 0$, we have for almost all $\omega \in \Omega$ that there is an $n_0 \in \N$ such that 
$n^{-1} \sum_{i=1}^n ( Z_{n,i}(\omega) - \tZ_{\delta,i}(\omega) )_+ = 0$ for all $n \ge n_0$. Hence $D_n$ converges almost surely to zero, and we have proven strong consistency.
\end{proof}


When dealing with the spatial sign covariance matrix thoroughly, the possibility of multiple instances of $X_i = t_n$ causes some nuisance. In the proof of Theorem \ref{th:cons}, this is covered implicitly, since $X_i = t_n$ implies $i \in \calI^C_n$. For Theorems \ref{th:an1} and \ref{th:an2}, we state the following lemma.
\begin{lemma} \label{lem:1}
Under conditions (\ref{num:an1:I}) and (\ref{num:an1:II}) of Theorem \ref{th:an1}, $(n-n^*)/\sqrt{n} \cip 0$, where $n^*$ is defined in Section \ref{sec:intro}.
\end{lemma} 
\begin{proof}
Recall that throughout the appendix, we let $t=0$. Let $A = \sup_n E(\!\sqrt{n} |t_n|)^4$, $B = E(|X|^{-3/2})$  and $1/8 < \alpha < 1/6$. With Markov's inequality we have with condition (\ref{num:an1:I}) of Theorem \ref{th:an1}, 
$P(|t_n| > n^{\alpha-1/2}) = P(\sqrt{n}|t_n| > n^\alpha)     \le A/n^{4\alpha}$, and with condition (\ref{num:an1:II}), 
$P( |X| < n^{\alpha-1/2}) = P(|X|^{-1}     > n^{1/2-\alpha}) \le B n^{3\alpha/2 - 3/4 }$. Thus
\[
	E\left[ \frac{n - n^*}{\sqrt{n}} \right] 
	\ = \ \frac{1}{\sqrt{n}} \sum_{i=1}^n E \Ind{X_i = t_n} 
	\ = \ 
	\frac{1}{\sqrt{n}} \sum_{i=1}^n  E \left[ \Ind{  X_i = t_n,\, |t_n| >   n^{\alpha-1/2} } \right]
	+ 
	\frac{1}{\sqrt{n}} \sum_{i=1}^n  E \left[ \Ind{  X_i = t_n,\, |t_n| \le n^{\alpha-1/2} } \right]
\]\[
	\le
	\frac{1}{\sqrt{n}} \sum_{i=1}^n  E \left[ \Ind{ |t_n| > n^{\alpha-1/2} } \right]
	+ 
	\frac{1}{\sqrt{n}} \sum_{i=1}^n  E \left[ \Ind{  |X_i| \le n^{\alpha-1/2} } \right]
	\ \le  \
	A n^{-4\alpha + 1/2} 	+  B n^{3\alpha/2 - 1/4}, 
\]
which converges to zero since $1/6 < \alpha < 1/8$. 
\end{proof}


\begin{proof}[Proof of Theorem \ref{th:an1}]
We show $\sqrt{n} \{\Sn(t_n)-\Sn\}\cip 0$ for $n\to \infty$. Due to Lemma \ref{lem:1} we may assume without loss of generality that $X_i \neq t_n$ for all $n, i \in \N$. Let, as before, $\Gamma_i = s(X_i)s(X_i)^T$ and $\Gamma_i(t_n) = s(X_i-t_n)s(X_i-t_n)^T$.
Using the identity
\be \label{eq:key_id}
 	 \Gamma_i(t_n) \ = \ 
 	  \Gamma_i  \ + \  |X_i|^{-2} \cdot \left( t_n t_n^T - X_i t_n^T - t_n X_i^T  \right) \ +\   \frac{ 2 X_i^T t_n  -t_n^T t_n}{|X_i|^2} \cdot \Gamma_i(t_n)
\ee
we have
\[
	\sqrt{n} \{\Sn(t_n)-\Sn\} \ = \ \frac{1}{\sqrt{n}}\sum_{i=1}^n \left( \Gamma_i(t_n) - \Gamma_i  \right)
\]
\[  
	= \ \frac{1}{\sqrt{n}}\sum_{i=1}^n\frac{t_n t_n^T}{|X_i|^2} 
	 \, - \, \frac{1}{\sqrt{n}}\sum_{i=1}^n\frac{X_i t_n^T}{|X_i|^2}
	 \, - \, \frac{1}{\sqrt{n}}\sum_{i=1}^n\frac{t_n X_i^T}{|X_i|^2}
	 \, + \, \frac{1}{\sqrt{n}}\sum_{i=1}^n \frac{2X_i^T t_n}{|X_i|^2}\Gamma_i(t_n)
	 \, - \, \frac{1}{\sqrt{n}}\sum_{i=1}^n \frac{t_n^T t_n}{|X_i|^2}\Gamma_i(t_n).
\]
Call the summands $A_n$, $B_n$, $C_n$, $D_n$ and $E_n$ from left to right. We will show convergence in probability to zero for each summand. 
For $B_n$ and $C_n$ we have
\[
	\|B_n\|\ =\ \|C_n\| 
	\ \leq \ 
	\left|	\sqrt{n}\, t_n \right| \cdot \left| \frac{1}{n}\sum_{i=1}^n \frac{X_i}{|X_i|^2} \right|
\]
The first factor is bounded in probability by assumption (\ref{num:an1:I}), and the second tends to 0 almost surely because of Kolmogorov's strong law of large numbers and assumption (\ref{num:an1:III}). 
For $A_n$ we have
\[
	\| A_n \|  \ \leq  \| n\, t_n t_n^T \| \cdot \left( \frac{1}{n^{3/2}} \sum_{i=1}^n {\frac{1}{|X_i|^2}} \right).
\]
Marcinkiewicz's strong law of large numbers guarantees convergence of the second factor to 0 almost surely as long as $E(|X|^{-2})^{2/3} < \infty$, which is fullfiled by assumption (\ref{num:an1:II}). 
Since $\|\Gamma_i(t_n)\| = 1$, the same argument applies to $E_n$.  
For $D_n$, we apply again identity (\ref{eq:key_id}) and obtain 
\begin{eqnarray*}
	D_n 
	  & = & \frac{2}{\sqrt{n}}\sum_{i=1}^n  \frac{X_i^T t_n}{|X_i|^4}\, X_i X_i^T
		\ + \ \frac{2}{\sqrt{n}}\sum_{i=1}^n  \frac{X_i^T t_n}{|X_i|^4}\, t_n t_n^T
		\ - \ \frac{2}{\sqrt{n}}\sum_{i=1}^n  \frac{X_i^T t_n}{|X_i|^4}\, t_n X_i^T \\
		& - & \frac{2}{\sqrt{n}}\sum_{i=1}^n  \frac{X_i^T t_n}{|X_i|^4}\, X_i t_n^T
		\ +	\ \frac{4}{\sqrt{n}}\sum_{i=1}^n  \frac{\left(X_i^T t_n\right)^2}{|X_i|^4}\, \Gamma_i(t_n)
		\ - \ \frac{2}{\sqrt{n}}\sum_{i=1}^n  \frac{\left(X_i^T t_n\right)\left( t_n^T t_n\right)}{|X_i|^4}\, \Gamma_i(t_n).
\end{eqnarray*}
Thus $D_n$ can be split up into six sums, which we denote by $D_{1,n},\ldots,D_{6,n}$ from left to right in the formula above. The matrix ${D_1}_n$ is best shown to converge to zero in probability by treating it element-wise. The $(l,m)$ element of $D_{1,n}$, $1 \le l,m \le p$, can be written as 
\[
	\left(D_{1,n}\right)_{l,m} \ = \ 
	2 \sum_{k=1}^p 
	\left( 	\sqrt{n}\, t_n^{(k)} \right) \cdot 
	\left( \frac{1}{n }\sum_{i=1}^n \frac{X_i^{(l)} X_i^{(m)} X_i^{(k)}} {|X_i|^4} \right).
\]
In each of the $p$ summands, the first factor is bounded in probability by assumption (\ref{num:an1:I}), and the second converges to zero almost surely, due to assumption (\ref{num:an1:III}) and Kolmogorov's strong law of large numbers. The remaining terms $D_{2,n},\ldots,D_{6,n}$ are shown to converge to zero in probability by Marcinkiewicz's strong law of large numbers. For the matrix norms of $D_{3,n}$, $D_{4,n}$ and $D_{5,n}$ we find the same upper bound as for $A_n$ above.  For $D_{2,n}$, we obtain the upper bound
\[
	\| D_{3,n} \| \ \le \ 2 n^{3/2}|t_n|^3 \cdot \left( \frac{1}{ n^2 } \sum_{i=1}^n {\frac{1}{|X_i|^3}} \right),
\]
where the first factor is again bounded in probability and the second converges to zero by Marcinkiewicz's strong law of large numbers, since by assumption (\ref{num:an1:III}), $E(|X|^{-3})^{1/2} < \infty$. The same argument applies to $D_{6,n}$, and the proof is complete. 
\end{proof}

			
\begin{proof}[Proof of Theorem \ref{th:an2}]
We use the same notation and the same decomposition of $\sqrt{n} \{\Sn(t_n)-\Sn\}$ as in the proof of Theorem \ref{th:an1}, and also assume that $X_i \neq t_n$ for all $n, i \in \N$. Of the altogether ten summand this expression has been split up into, seven, namely $A_n$, $D_{2,n}$, $D_{3,n}$, $D_{4,n}$, $D_{5,n}$, $D_{6,n}$, $E_n$ (those terms that contain $t_n$ in a multiplicity higher than 1), are as before shown to converge to zero by means of Marcinkiewicz's strong law of large numbers. We denote these seven terms by $R_n(\X_n,t_n)$ and focus on the remaining three, $B_n$, $C_n$, and $D_{1,n}$. We have 
\[
	\sqrt{n}\vec\, \{\Sn(t_n) - S\} 
	\, = \, 
	\sqrt{n} \vec\, \{\Sn- S\}
	 +  
	\vec\left\{
			\frac{2}{\sqrt{n}}\sum_{i=1}^n \frac{X_i^T t_n}{|X_i|^4} X_i X_i^T
		-	\frac{1}{\sqrt{n}}\sum_{i=1}^n \frac{X_i t_n^T + t_n X_i^T}{|X_i|^2}
		+ R_n(\X_n,t_n)
	\right\}
\]
\begin{eqnarray*}
	& = & 
	\sqrt{n} \vec\, \{\Sn- S\}
	 \ + \  \left\{ \frac{2}{n} \sum_{i=1}^n \frac{(X_i \otimes X_i) X_i^T }{|X_i|^4}\right\} \left(\sqrt{n}\, t_n\right) 
	 \ - \ \left(\sqrt{n}\, t_n\right) \otimes \left\{ \frac{1}{n} \sum_{i=1}^n \frac{X_i}{|X_i|^2}\right\} \\
	&   & 
	 \ - \ \left\{ \frac{1}{n} \sum_{i=1}^n \frac{X_i}{|X_i|^2}\right\} \otimes \left(\sqrt{n}\, t_n\right)
	 \ + \ \vec R_n(\X_n,t_n).
\end{eqnarray*}
This has the same asymptotic distribution as 
\[
\sqrt{n} \vec\, \{\Sn - S\}
	 \ + \  2 E \left\{\frac{(X \otimes X) X^T }{|X|^4}\right\} \left(\sqrt{n}\, t_n\right) 
	 \ - \ \left(\sqrt{n}\, t_n\right) \otimes E \left\{ \frac{X}{|X|^2}\right\} 
	 \ - \ E \left\{ \frac{X}{|X|^2}\right\} \otimes \left(\sqrt{n}\, t_n\right),
\]
which is a linear function of the random $p(p+1)$ vector $\xi_n =  \sqrt{n} ( t_n,  \vec\, \{\Sn - S\} )$. This function is given by $A \xi_n$ with $A$ being defined in Theorem \ref{th:an2}. Hence, by the continuous mapping theorem, $\sqrt{n}\vec\, \{\Sn(t_n) - S\} $ converges in distribution to $N_{p(p+1)}(0, A \Xi A^T)$, and the proof is complete. 
\end{proof}


\section*{Acknowledgment} The authors are very thankful to Roland Fried for many fruitful research discussions and a thorough proofreading of the manuscript. We are also indebted to the referee and the associate editor, whose thoughtful comments helped to improve the revised version of the article. Alexander D\"urre and Daniel Vogel were supported in part by the Collaborative Research Grant 823 
of the German Research Foundation. David Tyler was supported in part by the NSF Grant DMS-0906773.


\bibliographystyle{abbrvnat}

\end{document}